\documentclass[10pt,a4paper,reqno]{amsart}
\usepackage{amsfonts,amsthm,latexsym,amsmath,amssymb,amscd,amsmath,epsf, 
graphicx}

\newtheorem{tm}{Theorem}
\newtheorem{defi}[tm]{Definition}
\newtheorem{rem}[tm]{Remark}

\newtheorem{lm}[tm]{Lemma}

\newtheorem{cor}[tm]{Corollary}
\newtheorem{prop}[tm]{Proposition}

\newcommand{\al}{\alpha}

\newcommand{\la}{\lambda}

\newcommand{\bC}{\mathbb C}
\newcommand{\bR}{\mathbb R}

\newcommand{\C}{\mathcal C}

\newcommand{\eps}{\epsilon}

\newcommand{\field}[1]{\mathbb{#1}}
\newcommand{\R}{\field{R}}

\newcommand {\CC} {\mathcal C}

\begin{document}
\title
[On global non-oscillation of linear ode]{On global non-oscillation  of linear ordinary  differential equations with polynomial coefficients}

\author[D.~Novikov]{Dmitry Novikov }
\noindent 
\address{Department of Mathematics, Weizmann Institute of Science, Rehovot, 76100 Israel}
             \email{dmitry.novikov@weizmann.ac.il}

\author[B.~Shapiro]{Boris Shapiro}
\noindent 
\address{Department of Mathematics, Stockholm University, SE-106 91
Stockholm,
         Sweden}
\email{shapiro@math.su.se}

\begin{abstract}  In this  note we show that a  linear ordinary differential equation with polynomial coefficients  is globally non-oscillating in $\bC P^1$ if and only if it is Fuchsian, and at every its singular point any two  distinct characteristic exponents  have  distinct real parts.  As a byproduct of our study, we obtain a new explicit upper bound for the number of zeros of exponential polynomials in a horizontal strip. 
 \end{abstract}

\subjclass[2010] {Primary 34M03, Secondary 	34M10}

\keywords{Fuchsian differential equations, global non-oscillation, disconjugacy domain}

\date{}
\maketitle 

\section{Introduction} 
Let us recall the classical notions of disconjugacy and non-oscillation of a linear ordinary differential equation, see e.g. \cite {Co}. 
\begin{defi} {\rm A linear ordinary differential equation of order $k$ 
\begin{equation}\label{eq:main}
a_k(z)y^{(k)}+a_{k-1}(z)y^{(k-1)}+...+a_0(z)y=0,
\end{equation}
with continuous coefficients $a_j(z),\; j=0,\dots , k$ 
defined in a neighborhood of some simply-connected subset $I$ of $\R$ or $\bC,$ is called {\it disconjugate} (resp. {\it non-oscillating}) in $I,$ if every its nontrivial solution has in $I$ at most $k-1$ zeros (resp. finitely many zeros)  counted with multiplicities.}
\end{defi}

Observe that every equation \eqref{eq:main}   is disconjugate in any  sufficiently small interval in $\bR$ (resp. any sufficiently small  disk in $\bC$) centered at an arbitrary point $z_0\in \bR$ (resp. $z_0\in \bC$)  such that  $a_k(z_0)\neq 0.$ Analogously, every equation \eqref{eq:main}   is non-oscillating in any compact simply-connected set free from the roots of $a_k(z)$.  

\medskip
The study of different aspects and criteria of disconjugacy  and non-oscillation has been an active topic   in the past. While there exist satisfactory criteria of disconjugacy for the second order equations, the situation with the higher order equations is more complicated.  A number of necessary/sufficient conditions of disconjugacy for subsets of $\bR$ and $\bC$ are known in the literature mostly dating back at least four decades, see e.g. \cite {Ne}, \cite{La2}, \cite{Lev}. 
In the case of equations of order $2,$ disconjugacy is closely related to Sturm separation theorems; for higher order equations there is a related version of multiplicative Sturmian theory developed in \cite{Sh}. 

\medskip 
In this paper, for   a linear differential equation with polynomial coefficients, we  introduce the notion of its global non-oscillation  in  $\bC P^1$ by which we mean its classical non-oscillation in  an arbitrary {\it open contractible} domain obtained after the removal from $\bC P^1$ of an appropriate cut connecting all the singular points.  
Although oscillation/non-oscillation in the complex domain have been studied  since the  1920's,  (see e.g. \cite {Hi}), the notion of global non-oscillation seems to be new.  As an experienced reader can easily guess, the main motivation for our consideration comes from the second part of Hilbert's 16th problem. 

\medskip
 Consider a linear homogeneous differential equation  
\begin{equation}\label{eq:pol}
P_k(z)y^{(k)}+P_{k-1}(z)y^{(k-1)}+...+P_0(z)y=0,
\end{equation} 
with polynomial coefficients $P_k(z), P_{k-1}(z),\dots,P_0(z)$,  and $GCD(P_k,P_{k-1},\dots, P_0)=1$. Let $S$ be the set of all singular points of \eqref{eq:pol} in $\bC P^1$, i.e.,  the set of all roots of $P_k(z)$ (together with  $\infty$ if some of the limits $\lim_{z\to \infty} z^jP_{k-j}(z)/P_k(z),$ $j=0,\dots, k$ is infinite). For a given equation~\eqref{eq:pol}, let $d$ denote the cardinality of $S$.  

\begin{defi}{\rm   A system $\overline \CC:=\{\CC_j\}_{j=1}^{d-1}$ of smooth Jordan curves in $\bC P^1,$ each of them connecting a pair of distinct singular points, is called {\it an admissible cut} for equation \eqref{eq:pol} if and only if:   a) for any $i\neq j$, the intersection  $\CC_i\cap \CC_j$ is either empty or consists of their common endpoint; b) the union $ \cup_{j=1}^{d-1} \CC_j$ is topologically a tree in $\bC P^1$, i.e., the complement $\bC P^1\setminus \cup_{j}\CC_j$ is contractible; c) each $\CC_j$ has a well-defined tangent vector at each of its two endpoints. 
}
\end{defi} 

In particular, there exist admissible cuts consisting of straight segments connecting the singular points of \eqref{eq:pol}. 

\begin{defi} {\rm Equation \eqref{eq:pol} is called {\it globally non-oscillating} if, for any its admissible cut $\overline \CC$,   every  its nontrivial solution has finitely many zeros in   $\bC\setminus \overline \CC.$  } 
\end{defi}

The main result of this paper is the following criterion of  global non-oscillation.  

\begin{tm}\label{th:main} 
Equation \eqref{eq:pol} is globally non-oscillating if and only if: 
\begin{itemize}
\item[(i)]  it is Fuchsian;
\item [(ii)] at each  singular point all distinct characteristic exponents have pairwise distinct real parts.
\end{itemize} 
\end{tm}

\begin{rem} \label{rm:im}
{\rm One can easily notice that \eqref{eq:pol} is globally non-oscillating   if and only if some (and therefore any) domain  $\bC P^1\setminus \overline \CC$ can be covered by  finitely many open disconjugacy domains. 
 Observe that if one knows such a covering, then one gets an immediate upper bound for the total  number of zeros of nontrivial  solutions of \eqref{eq:pol} in  $ \bC P^1\setminus \overline \CC$. Namely, if the number of  open disconjugacy  domains covering $\bC P^1\setminus \overline \CC$ equals  $l,$ then 
any nontrivial solution of \eqref{eq:pol} has there at most $(k-1)l$ zeros counted with multiplicitties}.
\end{rem}

\medskip
In view of Remark~\ref{rm:im} the following problem is of fundamental importance. 

\medskip \noindent
{\bf Main Problem.}
Given an arbitrary equation~\eqref{eq:pol} satisfying the assumptions of Theorem~\ref{th:main}, estimate from above the number of disconjugacy domains which can form  an open covering of $\bC P^1\setminus \overline \CC$, for some admissible cut $\CC$.

\medskip
Observe that in  case of a Schr\"odinger equation  
$$-y^{\prime\prime} +P(z)y=0$$
with a polynomial potential $P(z),$  there is a classical  construction of  such coverings using the Schwarzian derivative of two linearly independent solutions of the latter equation which goes back to R.~Nevanlinna, \cite{Nev}.

\begin{rem} \label{rm:im}
{\rm
Let us also mention that Proposition~\ref{prop:constant} below, which is an important technical tool  used to prove Theorem~\ref{th:main}, is a new result in the classical area of the upper bounds for  the number of zeros of exponential polynomials and, therefore,  it is of independent interest. Such upper bounds are required in a wide range of mathematical disciplines, from applied mathematics to number theory. Essential progress in this area has been made in the 70's in the papers \cite{Tij}, \cite{VVT}, \cite{Vo}. But, to the best of our knowledge, in all the previous literature one only considered compact subdomains in $\bC$, mainly disks and rectangles, while Proposition~\ref{prop:constant} considers the case of an infinite strip. 
}
\end{rem}

\medskip 
\noindent 
{\em Acknowledgements.} The second author is grateful to the  Department of Mathematics and Computer Science of the  Weizmann Institute of Science for the  hospitality in January 2010 and February 2015 when this project was initiated and carried out.  The first author wants to thank G.~Binyamini for many discussions of the upper bounds of the number of zeros of Fuchsian equations over the years. 

\section{Proofs}

Our proof of Theorem~\ref{th:main} consists of several steps. 

\medskip\noindent
Step 1. The necessity of Conditions (i) and (ii) for  global non-oscillation of equation~\eqref{eq:pol}. 

Indeed, if \eqref{eq:pol} has a non-Fuchsian singularity at $p\in \bC P^1$, then, for any sufficiently small $\eps>0,$ almost any solution of  \eqref{eq:pol} has infinitely many zeros in the $\eps$-neighborhood of $p$ with a removed straight segment connecting $p$ with some point on the bounding circle. This property  contradicts to global non-oscillation.      
To finish Step 1,  consider a Fuchsian singularity of  \eqref{eq:pol} with two distinct characteristic exponents of the form $a+b_1I$ and $a+b_2I$. Then there exists a solution of \eqref{eq:pol} with the leading term $z^{a+\frac{(b_1+b_2)I}{2}}\cos \left(\frac{b_1-b_2}{2}\ln z\right).$
 Such a solution  has infinitely many zeros accumulating to $p$ which are located close to the horizontal  line passing through $p$.  This again  contradicts to global non-oscillation.     

\medskip \noindent
Step 2. Reduction to small neighborhoods of singular points. 

For any sufficiently small $\eps>0,$ construct a simply-connected domain $U_\eps\subset \bC P^1$ by: a) taking the large disk  $\{|z|<\eps^{-1}\}$ with  the $\eps$-neighborhoods of all zeros of 
$P_k$ removed, b) making cuts by straight segments between the bounding circles so that the obtained domain becomes contractible. 

\medskip 
The following complex analogue of the classical de la Vall\'ee Poussin theorem \cite{poussin} is proved in \cite[Theorem 2.6, Corollary 2.7]{Y}.

\begin{lm}\label{lm:ValPous}
Consider a homogeneous
\emph{monic} linear ordinary differential equation with holomorphic
coefficients
\begin{equation*}
    y^{(k)}+a_{k-1}(t)\,y^{(k-1)}+\cdots+a_0(t)\,y=0, \qquad
    t\in \bC.
\end{equation*}
 Then the variation of the argument of any solution $y(t)$ along a circular
arc $\gamma$ of a known length is explicitly bounded in terms of the
\emph{uniform upper bounds} $A_i=\sup_{t\in\gamma}|a_i(t)|$,
$i=0,\dots,k-1$. 
\end{lm}

Lemma~\ref{lm:ValPous} implies an explicit upper bound $B(\eps)$ for the number of zeros of any solution of  \eqref{eq:pol} in $U_\eps$. More exactly, the upper bound will depend 
on the upper bounds on the restrictions of $a_j$ to $\partial U_\eps$. The latter are polynomial in $\eps^{-1}$ if the coefficients $a_j$ are polynomials, so 
the upper bound is also polynomial in $\eps^{-1}$.

\begin{rem} {\rm Observe that, for any admissible system of cuts $\overline \CC$ and any sufficiently small $\eps$, the domain $\bC P^1\setminus \overline \C$ can be 
covered by 
finitely many $U_\eps$ (choosing different straight lines connecting the bounding circles) and finitely many  sectors of finite radii  centered at the singular points of \eqref{eq:pol}. This observation reduces  the proof of Theorem~\ref{th:main}  to providing finite upper bounds for the 
number of zeros of 
solutions of \eqref{eq:pol} in these sectors, see below.}
\end{rem}

\medskip \noindent
Step 3.  Equations with constant coefficients. ("Reduction``  to the case of equations with constant coefficients in a neighborhood of a Fuchsian singularity is obtained by using the logarithmic chart centered at  the singularity.  See also Steps 4-5.)

\begin{prop}\label{prop:constant}
 For any $\al\ge 0$ and for any equation 
\begin{equation}\label{eq:const}
EQ:\quad  a_ky^{(k)}+a_{k-1}y^{(k-1)}+...+a_0y=0,\; a_j\in \bC,\; a_k\neq 0
\end{equation}
such that all its  distinct characteristic roots  have distinct real parts, 
\begin{enumerate}
\item there exists an upper bound   $\sharp(EQ,\al)$   for  
the number of zeros   of all nontrivial solutions of \eqref{eq:const} in the horizontal strip $\{\Pi_\al:  | \Im(z)\le 
\al|\}$. (Here zeros are 
counted  with multiplicities.) 
\item in the generic case when all roots $\la_j,\; j=1,\dots , k$ of the characteristic equation of (\ref{eq:const}) are 
simple, with $\Re\la_1<\Re\la_{1}<\dots <\Re\la_k$, we get  
\begin{equation}\label{eq: simple exponents bound}
\sharp(EQ,\al)\le (k-1)^2+\frac 2 \pi 
(k-1)\mathcal{L}(\operatorname{EQ})\left[\al(\Xi+2)+\Theta \log4\right],
\end{equation}
where $\mathcal{L}(\operatorname{EQ})$ is the length of the shortest polygonal path
passing through all $\la_j$ and
\begin{equation*} 
\Theta:= \max_{1\le j \le  k-1}{|\Re 
(\la_j)-\Re(\la_{j+1})|}^{-1}, \qquad \Xi:=\max_{1\le j\le 
k-1}\left|\frac{\Im\la_j-\Im\la_{j+1}}{\Re\la_j-\Re\la_{j+1}}\right|.
\end{equation*}
\end{enumerate}
\end{prop}

The case of multiple characteristic roots will be considered in Step 4. 

\medskip
Our approach to the proof of Proposition~\ref{prop:constant} is  inspired by the Wiman-Valiron theory, see \cite{Va}. The main construction below  has a strong resemblance with the notion of a tropical polynomial in the modern tropical geometry.  The proof itself is based on a rather long sequence of lemmas and the last argument is given at the end of Step 5. 

\medskip
The general solution of \eqref{eq:const} is given by:  
\begin{equation}\label{eq:qpl resonant}
y=\sum_j A_j(z)e^{\la_jz},\quad\text{where}\, \deg A_j(z)=n_j, 
\sum (n_j+1)=k.
\end{equation} 
Define the {\it domain of a single term 
$y$-dominance  in $\Pi_\al$} as 
\begin{equation}\label{eq:def G(y, al)}
G(y,\al):=\{z\in \Pi_\al\; |\;\exists j=j(z),\,\exists \eps>0\;:\: | 
A_j(z)e^{\la_jz}|\ge (1-\eps) \sum_{i\neq  j}|A_i(z) e^{\la_iz}|\}.
\end{equation}

Note that  $G(y,\al)$ may contain at most $\;\min n_j\le k$ zeros of $y$, 
namely 
the common zeros of all $A_j(z).$ In particular, $G(y,\al)$ contains no zeros of $y$ at all in the case 
of simple characteristic exponents.

\begin{lm}\label{lm:mult const}
The complement  $\Pi_\alpha\setminus G(y, \al)$  can be 
covered by at most 
$k+k^2+k^3$ horizontal boxes (of height $2\al$) of 
the total width  not exceeding 
$$k^2(k+1)(4\Theta\ln k+4\al\Xi+4\al) + 8k^2\Theta.$$
\end{lm}

We first consider the case of  simple characteristic exponents $\la_j$.  
This case is  more transparent and the resulting estimates seem to be 
of correct order of magnitude. In this case the polynomials $A_j(z)$  are 
constants and will be denoted by $a_j$.

\begin{lm}\label{lm:impPi} In the case of simple characteristic exponents 
$\la_j$, the complement $\Pi_\al\setminus G(y,\al)$  can be 
covered by at most 
$k-1$ horizontal boxes (of height $2\al$) of 
the total width  not exceeding  
\begin{equation}
2\al(k-1)\,\Xi+2(k-1)\Theta\ln 4.
\end{equation}

\end{lm}

The principal case in Lemma~\ref{lm:impPi}  Êis  $\al=0$, i.e.  $\Pi_0=\R$.

\begin{lm}\label{lm:imp} In the above notations,  $\bR\setminus G(y,0)$ is contained in the union of at most $k-1$ closed intervals of the total length less 
than 
or equal to   
$2(k-1)\ln 4\cdot  \Theta.$
\end{lm}

 To prove Lemma~\ref{lm:imp}, we need an additional statement.  In $\bR^2$ with 
coordinates $(\mu,\phi)$ consider the $1$-parameter family  $\{Pt_j(u)\}_{j=1}^k$  of $k$ points  given by 
$$\mu=\Re (\la_j), \quad \phi=\ln |a_j e^{\la_ju}|,$$ where $u$ is a real-valued parameter. 
 For a given value of $u\in \bR,$ introduce the piecewise-linear function $\phi_u(\mu)$ as the {\it least concave majorant} of  $\{Pt_j(u)\}_{j=1}^k.$  By this we mean the minimal   concave function  $\phi_u(\mu)$ defined in the interval $[\Re(\la_1),\Re(\la_k)]$ such that all points $\{Pt_j(u)\}_{j=1}^k$ lie non-strictly below its graph, i.e. have their $\phi$-coordinate smaller than or equal to that of  $\phi_u(\mu)$.  (One can easily see that the graph of $\phi_u(\mu)$ is the upper part of the boundary of the convex hull of $\{Pt_j(u)\}_{j=1}^k$ connecting $Pt_1(u)$ and $Pt_k(u)$.)  Observe that, for any $u\in \bR$,  
\begin{equation}\label{eq:rel} \phi_u(\mu)=\phi_0(\mu)+u\mu.
\end{equation}

\begin{lm}\label{lm:extra} If, for $j=1,\dots , k-1,$
$$|\phi_u(\Re(\la_{j+1}))-\phi_u(\Re(\la_j))|\ge \ln 4,$$  
then $u\in 
G(y,0)$.  
\end{lm}

\begin{proof}[Proof of Lemma~\ref{lm:extra}] Define  the central index  of $\phi_u(\mu)$ by the formula: 
 $$i(u):=\{i\;| \Re(\la_i) \text{ is the point of the global maximum for } \phi_u(\mu)\},$$ 
comp. Ch. 1, \cite{Va}.  
Then for any $j\neq i,$ 
$$
|a_je^{\la_ju}|\le 
\exp(\phi_u(\Re\la_j))\le 4^{-|j-i(u)|}|a_{i(u)}e^{\la_{i(u)}u}|.$$ 
Therefore the 
inequality in the definition \eqref{eq:def G(y, al)} of $G(y,0)$ follows after the  
summation of a  geometric series. 
\end{proof}

\begin{cor}\label{cor:slope}  If $-u$ lies outside the $\ln 4 \cdot \Theta$-neighborhood of the 
set of all slopes of $\phi_0(\mu),$ 
then $u\in G(y,0)$.
\end{cor}

\begin{proof}  Formula~\eqref{eq:rel} implies that each slope of $\phi_u(\mu)$ 
equals the sum of the  respective slope of $\phi_0(\mu)$ and $u$. Therefore in 
the considered case, the absolute values of all  slopes of $\phi_u(\mu)$ exceed  
$\ln 4 \cdot \Theta, $ and 
the statement   follows immediately from Lemma~\ref{lm:extra}. \end{proof}

\begin{proof}[Proof of Lemma~\ref{lm:imp}] The $\ln 4 \cdot \Theta$-neighborhood of the set of slopes of 
$\phi_0(\mu)$  consists of the union of at most $k-1$ intervals of total 
length not exceeding $2(k-1)\ln 4 \cdot \Theta$. 
\end{proof}

\begin{proof}[Proof of Lemma~\ref{lm:impPi}]
Consider the general case of Lemma~\ref{lm:impPi} with $\al\ge 0$.

We repeat the above construction of Lemma~\ref{lm:imp} for $z$ running  along the horizontal line 
$\Im z= v$ with $|v|\le \al$.  
For every fixed $v$,     consider in $\bR^2$ with coordinates $(\mu,\phi)$, the $1$-parameter family  $\{Pt_j^v(u)\}_{j=1}^k$  of $k$ points 
  given by 
$$\mu=\{\Re (\la_j), \phi=\ln |a_j e^{\la_j(u+Iv)}|\}$$ where $u$ is a real parameter. 
 Introduce $\phi_u^v(\mu)$ as the {\it least concave majorant} of  $\{Pt^v_j(u)\}_{j=1}^k,$ for a given value of $u\in \bR$. Observe that, for any $u\in \bR$,  
\begin{equation}\label{eq:rel2} \phi^v_u(\mu)=\phi_0^v(\mu)+u\mu.\end{equation}

Now consider the set $Sl_\al:=\cup_{-\al\le v\le \al} \{k_j(v)\}$, where 
$k_j(v)$ are the slopes of $ 
\phi^v_0(\mu)$. We claim that  $Sl_\al$ is the union of at most $k-1$ closed 
intervals. Indeed, the set of slopes $\{k_j(v)\}$ changes
continuously with $v$, and consists of no more than $k-1$ points for each fixed $v$.

Moreover, as $\ln |a_j e^{I\la_jv}|=\ln|a_j|-v\Im\la_j$, the points 
$\{Pt_j^v(0)\}_{j=1}^k$ defining $\phi_0^v(\mu)$ depend linearly on $v$, 
namely they move up or down as $v$ changes.  The inequality 
$$  \left|\frac{\partial k_j(v)}{\partial v}\right|\le \Xi$$ 
is straightforward. 
Therefore the total length of $Sl_\al$ is at most 
$2\al\Xi(k-1)$.

By Corollary~\ref{cor:slope},  if $-u$ lies outside the $\ln 4 
\cdot \Theta$-neighborhood of 
$Sl_\al,$ then, for any $|v|\le \al$, $u+Iv$ lies in $G(y,\al)$  which settles   Lemma~\ref{lm:impPi}.
\end{proof}

\medskip\noindent

\medskip\noindent
Step $4$. Case of multiple characteristic exponents. 

In this case the dependence on $v$ of (analogs of) points  
$Pt_j^v(u)$ seems to be 
more complicated, and we are forced to consider the slopes of 
\emph{all} 
chords connecting these points, and not only those which lie on the boundary of their 
convex hull. This apparently 
leads to an excessive upper bound of the total width of $\Pi_\al\setminus G(y,\al)$.

\begin{proof}[Proof of Lemma~\ref{lm:mult const}]

Consider the absolute value $r_{jj'}$ of the ratio of any two terms in (\ref{eq:qpl resonant}). 
  The complement  $\Pi_\al\setminus G(y,\al)$ lies in 
the  union $\Sigma$ of the sets $\Sigma^o_{jj'}=\{|\ln 
r_{jj'}(z)|\le\ln k\}$, where $r_{jj'}$ is the absolute value of the ratio of 
two 
terms in (\ref{eq:qpl resonant}). 

\medskip
 We can write
\begin{equation}\label{eq:rjj}
\ln r_{jj'}= \ln |A_j/A_{j'}| -v\xi_{jj'}\theta_{jj'}+\theta_{jj'}u, 
\end{equation}
where 
\begin{equation*}
\theta_{jj'}=\Re(\la_j-\la_{j'}),\;\xi_{jj'}=\theta_{jj'}^{-1}\Im(\la_j-\la_{j'}).
\end{equation*} 

Set $W=\{|\Re(z-z_i)|\ge 4k\Theta\},$ where $z_i$ runs 
over all roots of all $ A_j$. Outside $W$ we have 
$$
\left|\frac{\partial}{\partial u}\ln |A_j/A_{j'}|\right|, 
\left|\frac{\partial}{\partial v}\ln |A_j/A_{j'}|\right| \le 
\frac{|\theta_{jj'}|}{2}.
$$

Additionally,   
$$
\Sigma^o_{jj'}\subset \Sigma_{jj'}=\{u+Iv\in\Pi_\al\;:\;
|\ln 
r_{jj'}(u)|\le\ln k+\al|\xi_{jj'}\theta_{jj'}|+\al|\theta_{jj'}|\},
$$
outside $W.$  
Note that 
$\Sigma_{jj'}$ is the union of boxes, since its definition is  independent of $v$. 

Therefore,  
\begin{equation}\label{eq:rjj'}
\left|\frac{\partial 
\ln r_{jj'}}{\partial u}\right|\ge \frac{|\theta_{jj'}|}2,
\end{equation}
outside $W$.

Thus $\Sigma_{jj'}$ intersects each connected component of 
$\R\setminus W$  in   an  interval of length at most $4|\theta_{jj'}|^{-1}\ln 
k+4\al|\xi_{jj'}| +4\al $. In other words,  
$\Sigma_{jj'}\setminus W$ is the union of  at most $k+1$ boxes of 
total 
width not exceeding $(k+1)(4\ln k|\theta_{jj'}|^{-1}+4\al|\xi_{jj'}|+4\al)$.

Taking the union over all possible pairs $(j, j')$, we conclude  that 
$\Sigma\setminus W$ lies in the  union of  at 
most $k^2(k+1)$ boxes of total width at most 
$k^2(k+1)(4\Theta\ln k+4\al\Xi+4\al) $. As $W\cap \Pi_\al$ is the union of at 
most $k$ boxes
of width at most $8k\Theta$ each, we obtain  that $\Sigma$ lies in the  
union of 
at most $k+k^2+k^3$ boxes of total width at most
$k^2(k+1)(4\Theta\ln k+4\al\Xi+4\al) + 8k^2\Theta$.
\end{proof}

Finally let us explain  how  Lemmas~\ref{lm:mult const} and \ref{lm:impPi}  imply  Proposition~\ref{prop:constant}. Consider the space  $QP_\Lambda=\{\sum_j A_j(z)e^{\la_j z}, 
A_j\in\bC[z]\}$ of dimension $k=\sum (1+\deg A_j)$ consisting of exponential polynomials, where 
$\Lambda=\{\la_j\}\subset \bC$ is some finite 
set.   The following 
result was proven in \cite{KY}. 
\begin{tm}[\cite{KY}]\label{th:KY}

The number of zeros of any function $f\in QP_\Lambda$ in a bounded convex domain $U$ does not exceed 
\begin{equation}
k-1+\frac 1 \pi \mathcal{L}(\Lambda)\operatorname{diam}(U),
\end{equation}
where $\mathcal{L}(\Lambda)$ is the length of a shortest polygonal path passing through all points of $\Lambda$.
\end{tm}

Theorem~\ref{th:KY} immediately implies an estimate on the number of zeros of $y$ in the 
boxes $B_j$ of Lemma~\ref{lm:mult const} and \ref{lm:impPi}. In the case of 
simple characteristic exponents (second part of Proposition~\ref{prop:constant})
$$
\sum\operatorname{diam} B_j\le 2(k-1)\left[\al\Xi+\Theta \log4\right]+4(k-1)\al,
$$
and  (\ref{eq: simple exponents bound})  follows.

\medskip \noindent 
Step $5.$ {Equation with non-constant coefficients in a 
semistrip.} 

In general, solutions of \eqref{eq:pol}  considered in the 
logarithmic 
chart  near its Fuchsian 
singularity have the  form 
\begin{equation}\label{eq: qpl with non-const coeff}
y=\sum_{j}\widetilde{A}_j(z)e^{\la_{j}z},
\end{equation} 
where  
$$
\widetilde{A}_j(z)=\sum_{r=0}^{n_j}a_{j,r}z^{n_j-r}(1+\eps_{j,r}),
$$
and $\eps_{j,r}$ is $2\pi I$-periodic,  $\eps_{j,r}=O(e^z)$ in any semistrip $\Pi_{\al,\beta}= 
\{|\Im z|\le \al, 
\Re z\le \beta\}$ for some  $\beta$ depending on \eqref{eq:pol} only.  To simplify our notation, let us assume that $\beta<0$.

\begin{lm}\label{lm:pert is bdd}
Assume that  $|\eps_{j,r}|<C e^{\Re z}$ in $\Pi_{\al,\beta}$.
Let $A_j(z)=\sum_{r=0}^{n_j}a_{j,r}z^{n_j-r}$ and $W$ be as in the proof of 
Lemma~\ref{lm:mult const}. Then
\begin{equation}
|\log|\widetilde{A}_j/A_j|\le C_{EQ,\al} \;\text{in}\; \Pi_{\al,\beta}\setminus W,
\end{equation}
where $C_{EQ,\al}$ is some constant depending on $\alpha$ and  \eqref{eq:pol} 
only.
\end{lm}

\begin{proof}
Let $\mathring{A}(z)=\sum |a_{j,r}||z|^r$. 
Evidently, $\mathring{A}_j(z)\le \prod_m (|z|+|z_m|)$, where $z_m$ are the roots of $A_j(z)$.
Also, $|\widetilde{A_j}-A_j|\le Ce^{\Re z} \mathring{A}_j(z)$.
Therefore,  for $z\in \Pi_{\al,\beta}\setminus W,$ we get 

\begin{eqnarray}
|\log|\widetilde{A}_j/A_j|\le Ce^{\Re z} \frac{\mathring{A}_j(z)}{|A_j(z)|}\le  Ce^{\Re z}\prod \frac{|z|+|z_m|}{\left||z|-|z_m|\right|}\le&\\\nonumber
\le  Ce^{\Re z} |z|^\ell \prod_{|z_m|<2|z|} \frac{1+|z_m|/|z|}{\left||z|-|z_m|\right|}
 \prod_{|z_m|>2|z|} \frac{|z/z_m+1|}{1-|z/z_m|}\le
Ce^{\Re z}|z|^\ell & \left(\frac 3 {4k\Theta}\right)^\ell 3^{k-\ell}.
\end{eqnarray}
Clearly, the latter function can be majorized by some number depending on $C,\al$ and $k,\Theta$ only. 
The constants $C, k, \Theta$ are determined by \eqref{eq:pol}. 
\end{proof}

\begin{rem}
Actually, dependence of $C_{EQ,\al}$ on $\al$ is very simple (as $O(\al^k)$ as $\al\to\infty$), but we do not need this.
\end{rem}

\begin{lm}\label{lm:pert}
In the above notation,  the zeros of $y$ in $\Pi_{\al,\beta}$ lie in at most $k+k^2+k^3$  boxes of total width at most 
$$k^2(k+1)(4\Theta\ln k+4\al\Xi+4\al+4C_{EQ,\al}) + 8k^2\Theta.$$
\end{lm}

\begin{proof}
We repeat the proof of   Lemma~\ref{lm:mult const}. Namely, consider the absolute value $\tilde{r}_{jj'}$ of the ratio of any two terms in 
\eqref{eq: qpl with non-const coeff}. 
  The complement  $\Pi_\al\setminus G(y,\al)$ lies in 
the  union $\Sigma$ of the sets $\tilde{\Sigma}^o_{jj'}=\{|\ln 
\tilde{r}_{jj'}(z)|\le\ln k\}$. 
But, according to Lemma~\ref{lm:pert is bdd}, $|\log \tilde{r}_{jj'}- \log 
r_{jj'}|\le C_{QE,\al}$, where $r_{jj'}$ was defined in the proof of 
Lemma~\ref{lm:mult const}.
So, it is enough to require $|\ln 
\tilde{r}_{jj'}(z)|\le\ln k +C_{EQ,\al}$, i.e. outside $W$
$$
\tilde{\Sigma}^o_{jj'}\subset \tilde{\Sigma}_{jj'}=\{u+Iv\in\Pi_\al\,:\,
|\ln 
r_{jj'}(u)|\le\ln k+\al|\xi_{jj'}\theta_{jj'}|+\al|\theta_{jj'}|+C_{EQ,\al}\}. 
$$ 
Repeating the same arguments as in Lemma~\ref{lm:mult const} with  $\tilde{\Sigma}_{jj'}$ instead of $\Sigma_{jj'}$, we arrive at the 
required estimates. 
\end{proof}
%
%
%
%
%
%
%

\begin{proof}[Proof of Theorem~\ref{th:main}]
Let $y^{(k)}+b_1(z)y^{(k-1)}+...+b_k y=0$ be the reduced form 
(=divided by its 
leading term) of \eqref{eq:pol} in the logarithmic chart near 
its 
Fuchsian singularity. Assume that $b_j(z)$ are bounded by $C$ in  
$\Pi_{\al,\beta}$ 
(The Fuchsian property implies that each $b_j(z)$ tends to some finite limit 
when $z\to\infty $ 
in $ \Pi_{\al,\beta}$).

Example in \cite{Y} immediately following after Corollary 2.7 of this paper,  implies that $y$ has at most $2(k+1)+\frac {k+1} 
{\log(9/4)}\ell C$ zeros  in  
$\Pi_{\al,\beta}$, where 
$$
\ell\le 2k^2(k+1)(4\Theta\ln k+4\al\Xi+4\al+4C_{EQ,\al}) + 
16k^2\Theta+4(k+k^2+k^3)\al,
$$
 is the total 
 perimeter of all boxes appearing in  Lemma~\ref{lm:pert}. 
 
After going back from logarithmic chart to the original coordinate, we obtain   an upper 
bound for the number of zeros of any solution of 
\eqref{eq:pol} in the sector 
$\{|z-p|\le e^\beta, |\arg z|\le\al\}$ at the Fuchsian 
singular point $p$. 

\end{proof}

The sequence of steps 1-5 settles Theorem~\ref{th:main}.

\end{document}